\documentclass{amsart}

\usepackage{enumerate}

\usepackage[backref=page]{hyperref}
\hypersetup{
    colorlinks=true,
    linkcolor=blue,
    filecolor=blue,      
    urlcolor=blue,
    citecolor=blue,
}
\usepackage{amsrefs}
\usepackage{amsmath}
\usepackage{amsthm}
\usepackage{amssymb}

\newcommand{\Ten}[2]{\mathcal{{N}}\left(#1,#2\right)}
\newcommand{\Id}[3][]{\mathcal{{I}}_{#1}\left(#2,#3\right)}
\newcommand{\Op}[3][]{\mathcal{{Z}}_{#1}\left(#2,#3\right)}

\usepackage{wasysym} 
\usepackage{xcolor}
\usepackage{tikz}
\usepackage{tikz-cd}
\usepackage[all]{xy}
\usepackage{ytableau}
\usepackage{url}
\usepackage[labelformat=simple]{subcaption}

\usetikzlibrary{arrows}

\usepackage{algorithm}
\usepackage[noend]{algpseudocode}
\usepackage{comment}

\makeatletter
\def\BState{\State\hskip-\ALG@thistlm}
\makeatother

\DeclareMathOperator{\vav}{\ell}

\numberwithin{figure}{section}

\numberwithin{equation}{section}

\newtheorem{thm}[equation]{Theorem}
\newtheorem*{thm*}{Theorem}
\newtheorem{lem}[equation]{Lemma}
\newtheorem{prop}[equation]{Proposition}
\newtheorem{cor}[equation]{Corollary}
\newtheorem{lemma}[equation]{Lemma}

\theoremstyle{definition}

\theoremstyle{remark}
\newtheorem{remark}[equation]{Remark}

\DeclareMathOperator{\End}{End}
\DeclareMathOperator{\Hom}{Hom}

\DeclareMathOperator{\Aut}{Aut}

\DeclareMathOperator{\Adj}{Adj}
\DeclareMathOperator{\Der}{{\rm Der}}
\DeclareMathOperator{\gl}{\mathfrak{gl}}
\DeclareMathOperator{\GL}{GL}

\DeclareMathOperator{\PGL}{PGL}

\newcommand{\bmto}{\rightarrowtail}
\newcommand{\F}{\mathbb{F}}

\newcommand{\la}{\langle}
\newcommand{\ra}{\rangle}

\renewcommand{\phi}{\varphi}
\renewcommand{\epsilon}{\varepsilon}

\newcommand{\bra}[1]{\langle#1|}
\newcommand{\ket}[1]{|#1\rangle}
\newcommand{\braket}[2]{\langle #1|#2\rangle}
\newcommand{\den}[1]{\Leftcircle\hspace*{-1mm}#1\hspace*{-1mm} \Rightcircle}
\newcommand{\dual}{*}

\renewcommand{\leq}{\leqslant}
\renewcommand{\geq}{\geqslant}
\renewcommand{\sl}{\mathfrak{sl}}

\newcommand{\op}{{\tiny op}}

\begin{document}

\title{Tensor Isomorphism by conjugacy of Lie algebras}
\author{Peter A. Brooksbank}
\address{
	Department of Mathematics\\
	Bucknell University\\
	Lewisburg, PA 17837
}
\email{pbrooksb@bucknell.edu}

\author{Joshua Maglione}
\address{
	Department of Mathematics\\
	Otto von Guericke University Magdeburg\\
   39106 Magdeburg, Germany
}
\email{joshua.maglione@ovgu.de}

\author{James B. Wilson}
\address{
	Department of Mathematics\\
	Colorado State University\\
	Fort Collins, CO 80523
}
\email{James.Wilson@ColoState.Edu}
\thanks{This work was partially supported by Deutsche Forschungsgemeinschaft
grant VO~1248/4-1 project number~373111162, National Science Foundation grants
DMS-1620362, DMS-1620454, and The Simons Foundation Grant \#636189.}

\date{\today}

\begin{abstract}
   We introduce an algorithm  
   to decide isomorphism between tensors. 
   The algorithm uses the Lie algebra 
   of derivations of a tensor 
   to compress the space in which the 
   search takes place to a so-called 
   {\em densor space}. To make the method 
   practicable we give a polynomial-time 
   algorithm to solve a generalization of
   module isomorphism for 
   a common class of Lie modules.
   As a consequence, we show that 
   isomorphism testing is in polynomial time 
   for tensors whose derivation algebras are 
   classical Lie algebras and whose densor spaces are 
   1-dimensional.
   The method has been implemented in the 
   {\sc Magma} computer algebra system.
\end{abstract}

\maketitle

\section{Introduction}
\label{sec:intro}

Techniques to decide isomorphism for algebraic structures 
such as groups, algebras, and modules often involve 
testing whether two multilinear maps (tensors) 
are equal under basis changes. 
Examples include isomorphisms tests for 
finite $p$-groups that work through the factors of
the exponent $p$-central series~\cites{ELGO:Automorphism, OBrien:Isomorphism}, 
and others that use more general 
filtrations~\cites{Wilson:FilterRefinements,
Maglione:CompatibleFilters}. A recent rethinking of these 
approaches led to an isomorphism test for graded algebras 
that identifies an optimal route through
the filtration~\cite{BOW}. In all of these techniques, 
the initial task is to decide isomorphism between tensors. 
This paper introduces a
new algorithm to solve tensor isomorphism by exploiting 
the action of the Lie
algebra of derivations on the vector space of tensors.
The algorithm is particularly well suited to isomorphism problems 
for highly symmetric structures, such 
as those found in \citelist{\cite{BGMN:Automorphism} 
\cite{Freedman:ExceptionalChevalley} \cite{GRS}}. 
We note that methods for generic tensors, algebras, and groups can be found in \citelist{
	\cite{BLQW} \cite{ELGO:Automorphism} \cite{Li-Qiao}}.

\subsection{Tensor isomorphism}
Fix a field $K$ and finite-dimensional $K$-vector spaces
$V_1,\ldots,V_{\ell}$. 
A \emph{tensor} is an element, $t$, of 
$(V_1\otimes\cdots \otimes
V_{\ell})^{\dual}=\hom_K(V_1\otimes\cdots\otimes V_{\ell}, K)$, 
and the integer $\ell\geqslant 2$ is its {\em valence}.
We interpret $t$
as a multilinear function $\bra{t}\colon V_1\times \cdots \times V_{\ell}\bmto
K$ ($\bmto$ to denote multilinear) that may be 
evaluated on inputs
$v=(v_1,\ldots,v_{\ell})\in V_1\times\cdots\times V_{\ell}$ using a Dirac styled
``bra-ket'' notation as follows:
\begin{align*}
	\braket{t}{v} = \braket{t}{v_1,\ldots,v_{\ell}}\in K.
\end{align*}
A tensor is also determined by its coordinates relative to bases
$\{e_{a1},\ldots,e_{ad_a}\}$ for each $V_a$. That is, for $a\in[\ell]=\{1,\dots,
\ell\}$ and $i_a\in [d_a]$, one specifies the scalar
\begin{align}\label{def:hypermatrix}
	T_{i_1\cdots i_{\ell}} & = \braket{t}{e_{1i_1},\ldots, e_{\ell i_{\ell}}}\in K.
\end{align}
As in elementary linear algebra, one can pass back and forth between the
``hyper-matrix'' $[T_i]_{i\in [d_1]\times \cdots \times [d_{\ell}]}\in
K^{d_1\times \cdots \times d_{\ell}}$ and the associated multilinear map
$\bra{t}$.  

For $\omega=(\omega_1,\ldots,\omega_{\ell})\in
\GL(V_1)\times\cdots\times\GL(V_{\ell})$ and $t\in (V_1\otimes \cdots \otimes
V_{\ell})^\dual$, define $t^{\omega}\in (V_1\otimes \cdots \otimes
V_{\ell})^\dual$ as follows:
\begin{align*}
	\langle t^\omega|v\rangle 
	& = \langle t|\omega v\rangle
	 = \langle t|\omega_1 v_1,\ldots,\omega_{\ell}v_{\ell}\rangle.
\end{align*}
Tensors $s$ and $t$ are {\em isomorphic} if $s^{\omega}=t$ for some $\omega \in
\GL(V_1)\times\cdots\times\GL(V_{\ell})$. Expressed in terms of hypermatrices,
this is a natural extension of equivalence of matrices up to row and column
operations. In a recent paper,
Grochow and Qiao prove that the problem of
deciding whether two tensors are isomorphic has  
connections with many familiar and difficult
decision problems, such as the graph and group isomorphism problems~\cite{GQ:Isomorphism}. 

Throughout the paper `algorithm' will mean \emph{Las Vegas algorithm}, which
always returns the correct answer but may, with bounded probability, abort
without an answer. We adopt an {\em arithmetic model} of computation, wherein
all field operations in $K$ have unit cost and are precise (no rounding). When
$K$ is infinite, we assume the existence of oracles to factor elements from
$\mathbb{Z}$ and from $\mathbb{Q}[x]$. Note that the product $\cdot:A\times
A\bmto A$ of a $K$-algebra can be treated as a tensor in $(A\otimes A\otimes
A^*)^*$. Hence, we assume tensors and algebras are both given by fixing bases
for all vector spaces involved, and specifying the 
scalars in \eqref{def:hypermatrix}. (In the
algebra context it is common to refer to the 
scalars as \emph{structure constants}.)

\subsection{The derivation-densor method}
\label{sec:intro-compress}

Write $\gl(V)$
for the Lie algebra on $\End(V)$ with Lie bracket given by the matrix
commutator. The {\em derivation algebra} of a tensor $t\in (V_1\otimes
\cdots \otimes V_{\ell})^\dual$ is the Lie algebra 
\begin{align*}
	\Der(t) & = \left\{ 
		\delta\in \bigoplus_{a=1}^{\vav} \gl(V_a)~\middle|~ \begin{array}{c} \forall a\in[\vav],\; \forall v_a\in V_a, \\ 
		\langle t |\delta_{1} v_{1},\ldots,v_{\vav}\rangle 
		+\cdots +
		\langle t | v_{1},\ldots, \delta_{\vav} v_{\vav}\rangle = 0
		\end{array}
	\right\}.
\end{align*}
Next we use a generalization of the standard Whitney tensor products called the
\emph{densor spaces} (short for \emph{derivation tensor spaces}) as introduced
in \cite{FMW:densors}.  Start with $\Delta\subset \gl(V_1)\times \cdots\times
\gl(V_{\vav})$ and define 
\begin{align*}
	\den{V_1,\ldots,V_{\vav}}_{\Delta} &= \left\{ s\in (V_1\otimes\cdots\otimes V_{\vav})^\dual 
		~\middle|~ \Delta\subseteq \Der(s)\right\}.
\end{align*}
The notation $\den{V_1,\ldots,V_{\vav}}_{\Delta}$ 
perhaps requires some explanation.
Consider the case $\vav=2$, and let 
$\delta\mapsto \delta^{\dagger}$ be 
a ring anti-isomorphism $\End(V_1)\to \End(V_1)^{\op}$. 
If $\hat{\Delta}=\{(-\delta_1^{\dagger},\delta_2)\mid (\delta_1,\delta_2)\in \Delta\}
\subset \End(V_1)^{\op}\times \End(V_2)$, then 
\begin{align*}
   (V_1\otimes_{\hat{\Delta}} V_2)^{\dual}
   & = \{t\in (V_1\otimes V_2)^{\dual} \mid \braket{t}{v_1(-\delta_1^{\dagger}),v_2}=\braket{t}{v_1,\delta_2 v_2}\}\\
   & = \{t\in (V_1\otimes V_2)^{\dual} \mid \braket{t}{\delta_1 v_1,v_2}+\braket{t}{v_1,\delta_2 v_2}=0\}
   = \den{V_1,V_2}_{\Delta}.
\end{align*} 
Thus, densor spaces are multivalent 
generalizations of $\otimes$. The notation of 
$\Leftcircle \Rightcircle$ is conceived as
backward and forward letters $D$---for derivation---but stylized to evoke the connection to the symbol $\otimes$.
We abbreviate $\den{V_1,\ldots,V_{\vav}}_{\Der(t)}$ to $\den{t}$.

Our algorithm for tensor isomorphism 
uses $\Der(t)$ and $\den{t}$ 
together with the subgroup of $\GL(V_1)\times \cdots \times
\GL(V_{\vav})$ that normalizes $\Der(t)$, namely 
\begin{align*} 
	N(\Der(t)) &= \left\{\omega \in \prod_{a=1}^{\ell}\GL(V_a)~\middle|~ \forall \delta\in \Der(t),(\omega_1^{-1}\delta_1\omega_1,\dots, \omega_{\vav}^{-1}\delta_{\vav}\omega_{\vav})\in\Der(t)\right\}.
\end{align*} 
Algorithm~\ref{algo:der-densor} gives a high level 
description of the isomorphism
test, which we call the {\em derivation-densor method}.

\begin{algorithm}
   \caption{(Derivation--Densor)}
   \label{algo:der-densor}
   \begin{algorithmic}[1]
   \Require Tensors $s,t\in (V_1\otimes\cdots\otimes V_{\vav})^*$.
   \Ensure $\omega\in G:=\GL(V_1)\times \cdots \times \GL(V_{\vav})$
   with $s^\omega=t$, if such exists.
   
   \State Compute $\Der(s)$ and $\Der(t)$.
   \If{$(\exists \mu \in G)(\Der(s)^{\mu}=\Der(t))$}
      \State Build the densor space $\den{t}$.
      \State Compute the action of $N(\Der(t))$ on $\den{t}$.
      \If{$(\exists \nu\in N(\Der(t)))(s^{\mu\nu} = t)$}
         \Return{$\omega:=\mu\nu$.}
      \Else
         ~Report $s\not\cong t$. 
      \EndIf
   \Else ~Report $s\not\cong t$. 
   \EndIf
   \end{algorithmic}
\end{algorithm}

Lines 1 and 3 are carried out by solving systems of linear equations. The
practicability of the method depends critically on the related problems in 
Lines 2, 4 and 5. These problems are known to be 
difficult in general, but if we have
sufficient control of $\Der(t)$ and its representation in
$\gl(V_1)\times\cdots\times\gl(V_{\ell})$, then we can carry out these tasks
directly in polynomial time (cf.~Theorem~\ref{thm:tensor-iso}). 

We note that existing methods \citelist{\cite{BW:tensor} \cite{BMW}
\cite{BMW:exact} \cite{LW:iso} \cite{Wilson:Skolem-Noether}} employ similar
strategies, but instead of Lie algebras of derivations they used associative
algebras $A_i$ of so-called \emph{adjoints}, and instead of the densor space
they compress the search space using 
more traditional tensors of the form $(V_1\otimes_{A_1}\cdots
\otimes_{A_{\ell-1}}V_{\ell})^*$. The associative approach became 
known as the 
{\em adjoint-tensor method}; details 
are given in Section~\ref{sec:method}.

The use of Lie algebras has a distinctive advantage over
adjoint-tensor methods.  Simple modules of associative algebras have
fixed dimensions; for example, $K^2$ is the only simple module of $\mathbb{M}_2(K)$.
Thus, if the adjoint algebras $A_i$ have bounded dimensions, then
$\dim (V_1\otimes_{A_1}\cdots \otimes_{A_{\ell-1}}V_{\ell})^*$ is proportional
to $\prod_{i=1}^{\ell}\dim V_i$.
As $\den{t}$ is a Lie module, its dimension is governed 
by quantities such as the Littlewood--Richardson coefficients.  
Even 
Lie algebras of bounded dimension, such as $\mathfrak{sl}_2(K)$,
can have simple modules of arbitrary dimensions, 
which means that $\dim \den{t}$ can be 
surprisingly small. Indeed, there are infinite families of tensors $t$ with $\dim
\den{t}=1$ such that, if $A_i$ are associative algebras 
satisfying $t\in(V_1\otimes_{A_1}\cdots \otimes_{A_{\ell-1}}V_{\ell})^*$, 
then $\dim
(V_1\otimes_{A_1}\cdots \otimes_{A_{\ell-1}}V_{\ell})^* 
=\prod_{i=1}^{\ell}\dim V_i$ (cf.~Theorem~\ref{thm:infinite-family} and Remark~\ref{rem:fast}).

As noted above, the performance of 
derivation-densor depends, for certain 
inputs, on 
difficulties faced in lines 2 and 4. 
It is often line 5, however, that presents 
the most serious challenges. Here, we search 
through the cosets of the centralizer
$C(\Der(t))=\{\omega\in N(\Der(t))\mid \delta\in \Der(t),\delta^\omega=\delta\}$
in $N(\Der(t))$.  
Although efficient 
module-theoretic techniques can often be used 
to refine this search,
for many tensors $t$ the 
derivation algebra
$\Der(t)=\{(-\alpha_1,\dots,-\alpha_{\vav-1},\alpha_1+\cdots+\alpha_{\vav})\mid
\alpha_a\in K\}$ is as small as it can possibly be, which
in turn means $N(\Der(t))/C(\Der(t))$ can be as 
large as $\PGL(V_1)\times\cdots \times \PGL(V_{\ell})$.  
Broadly speaking, derivation-densor works best when the given tensors 
possess enough global symmetry 
to be detectable by their derivation algebras, 
and such settings are indeed the focus of this paper.
However, since setting it up involves only linear 
algebra, derivation-densor serves as an efficient 
first reduction for many isomorphism problems.



\subsection{Using the derivation-densor method}
\label{sec:intro-use}
The key problems in Lines 2 and 4 of 
Algorithm~\ref{algo:der-densor} reduce to a variation 
of the module isomorphism problem for (non-associative)
algebras.
Write ${_{A}V}$ to
indicate that $V$ is a left $A$-module. We say modules ${_{A_1} V_1}$ and
${_{A_2} V_2}$ are {\em pseudo-isomorphic} if there is an algebra isomorphism
$\psi\colon A_1\to A_2$ and a $K$-linear isomorphism $\Psi\colon V_1\to V_2$
such that
\begin{align}\label{def:semi-linear}
	& (\forall x_1\in A_1)(\forall v_1\in V_1) & \Psi(x_1 v_1)=\psi(x_1)\Psi(v_1).
\end{align}
When $A=A_1=A_2$ and $\psi=1$ this is the 
usual notion of $A$-module isomorphism.
In that case, a polynomial-time 
algorithm of the first author 
and Luks~\cite{BL:mod-iso} may be used to build
$\Psi$ from units in the associative algebra
$\Hom_L(V_1,V_2)\Hom_L(V_2,V_1)\subset\End_L(V_1)$.  
When $\psi$ is allowed to vary, however, the problem 
becomes much more difficult.  
Indeed,
Grochow has shown that deciding Lie module
pseu\-do-i\-so\-morph\-ism is at least as 
hard as deciding graph isomorphism~\cite{Grochow:thesis}. 
There are, however, polynomial-time solutions for special classes,
such as modules of simple and cyclic associative algebras~\cite{BW:mod-iso}, and
simple modules of simple Lie algebras over $\mathbb{C}$~\cite{Grochow:thesis}.
The following theorem, which 
we consider to be of independent interest, 
supplements those
results (a Lie algebra $L$ has {\em Chevalley type} if $[L,L]$ has a
Chevalley basis).

\begin{thm}
\label{thm:Lie-conjugacy}
	Let $K$ be a field with either $K=6K$ finite or $K/\mathbb{Q}$ finite. There
	is a polynomial-time algorithm that decides pseudo-isomorphism of faithful simple
	fi\-nite-dim\-en\-sion\-al Lie modules over Lie algebras of Chevalley type.
\end{thm}

A tensor $t\in (V_1\otimes \cdots\otimes V_{\vav})^*$ is \emph{degenerate} if
there exists $0\neq v_a\in V_a$ such that, for all $b\neq a$ and all $v_b\in
V_b$, $\langle t | v_1,\dots, v_{\vav}\rangle = 0$. It is elementary to reduce
the isomorphism problem for arbitrary tensors to the nondegenerate case. The
polynomial-time algorithm in Theorem~\ref{thm:Lie-conjugacy} is vital to the
proof of the next theorem.

\begin{thm}
\label{thm:tensor-iso}
	Let $K$ be a field with either $K=6K$ finite or 
	$K/\mathbb{Q}$ finite, let $V_1,\ldots,V_{\vav}$ 
	be finite-dimensional $K$-spaces, and let 
	$d=\sum_{a=1}^{\vav}\dim_K V_a$.
	For nondegenerate $s,t\in (V_1\otimes
	\cdots\otimes V_{\vav})^*$ with $\Der(s)$ of 
	Chevalley type and $\dim\den{s}=1$, 
	one can decide using $d^{O(1)}$ steps 
	whether $s\cong t$.
\end{thm}

The tensors $t$ for which $\dim \den{t}=1$
are interesting special cases in their own right, and 
they are more common than one might suspect. 
In Section~\ref{sec:examples} we construct an 
infinite family of tensors, arising naturally from
the representation theory of classical Lie algebras, 
whose densor space is 1-dimensional.

\section{Algebraic tensor compression}
\label{sec:method}

The use of rings to decrease the dimension of the search 
spaces in tensor isomorphism is not new, but it has  
heretofore been developed only for associative 
rings. 
In this section we briefly describe the history of 
{\em algebraic tensor compression}, and how 
it led to the derivation-densor method.

\subsection{The adjoint-tensor method}
\label{sec:adj-ten}
The first tensor compression method was introduced 
in \cite{LW:iso} for the case $\ell=2$, and soon
after generalized in~\citelist{\cite{BW:tensor} \cite{BMW}
\cite{Wilson:Skolem-Noether}}. 
Given a bilinear map (bimap) $\bra{t}\colon V_1\times
V_2\rightarrowtail V_0$, its \emph{adjoint algebra} is
\begin{align*}
   \Adj(t) & =\{\alpha \in \End(V_1)\times \End(V_2)^{\mathrm{op}} \mid 
   \langle t|\alpha_2v_2, v_1\rangle=\langle t|v_2,v_1\alpha_1\rangle\},
\end{align*}
and its associated tensor space is $V_1\otimes_{\Adj(t)} V_2$.
Observe that
$\bra{t}$ naturally factors through the bimap
$\otimes_{\Adj(t)}\colon V_1\times
V_2 \bmto V_1\otimes_{\Adj(t)} V_2$. 

The {\em adjoint-tensor} method solves the isomorphism problem 
between bimaps $s$ and
$t$ by first deciding if there exists $\mu$ conjugating $\Adj(s)$ to $\Adj(t)$,
which is done by the methods of \citelist{\cite{BW:Isom} \cite{Ivanyos-Qiao}}.
Then we carry out a search within the compressed space
$\Hom(V_1\otimes_{\Adj(t)}V_2, V_0)$---in which 
both $s^\mu$ and $t$ now reside---under the action of the potentially much smaller group normalizing
$\Adj(t)$, modulo $\Adj(t)^{\times}$.

This process is captured concisely as follows:
\begin{align*}
	(\exists \phi)(s^\phi = t) 
	& \Longleftrightarrow (\exists \mu)(\exists\nu)
	\left\{\begin{array}{rcl}
		\Adj(s)^{\mu} & := & \mu^{-1}\Adj(s)\mu =  \Adj(t),  \\
		\Adj(t)^{\nu} & =  & \Adj(t), \text{ and} \\
		(s^\mu)^\nu & = & t\in \Hom(V_1\otimes_{\Adj(t)}V_2,V_0).
	\end{array}\right.
\end{align*}

The method distinguishes $V_0$ due to its role as the codomain.
One could, however, just as easily consider $s,t$ as tensors in $(V_1\otimes
V_2\otimes V_0^*)^*$. With this interpretation the compressed tensor
space is $(V_1\otimes_{\Adj(t)} V_2\otimes V_0^*)^*$, which now
seems like an arbitrary choice. To
reconcile the apparent asymmetry, 
the third author introduced a generalization
involving operations between all {\em pairs} of
 $V_a$ and $V_b$~\cite{Wilson:Skolem-Noether}. 
\noindent {\em The guiding principle of the derivation-densor 
algorithm is to move away from binary tensor products entirely.}

\subsection{A broader view}
\label{sec:gen-adj-ten}
Using an emerging theory of transverse operators 
on tensor spaces, one can generalize the 
adjoint-tensor method in many different ways. 
The theory, which is based on a ternary 
Galois correspondence, is developed in detail 
in the forthcoming work~\cite{FMW:densors}; we 
describe and prove here only what is needed 
for our isomorphism test.
 
Given $t\in
T:=(V_1\otimes\cdots\otimes V_{\vav})^*$, $f(X)=\sum_e \lambda_e X^e$ 
an element of the polynomial ring $K[X]:=K[x_1,\ldots,x_\ell]$, and $\omega\in\Omega:=
\End(V_1)\times \cdots \times \End(V_{\vav})$, 
define a new multilinear form 
$\langle t| f(\omega)$ as follows:
\begin{align*}
   \langle t| f(\omega) |v\rangle 
   & = \sum_e \lambda_e \langle t| \omega_1^{e_1}v_1,\ldots,
      \omega_\ell^{e_\ell}v_\ell\rangle.
\end{align*}
Given $S\subset T$, $P\subset K[X]$, and $\Upsilon\subset \Omega$, define three sets
\begin{align*}
   \Ten{P}{\Upsilon} & = \{ t\in T\mid \forall f\in P,\forall \omega\in \Upsilon,\forall v,\;
   \langle t|f(\omega)|v\rangle =0 \},\\
   \Id{S}{\Upsilon} & = \{ f\in K[X]\mid \forall t\in S,\forall \omega\in \Upsilon,\forall v,\;
   \langle t|f(\omega)|v\rangle =0 \},\\
   \Op{S}{P} & = \{ \omega\in \Omega\mid \forall t\in S,\forall f\in P,\forall v,\;
   \langle t|f(\omega)|v\rangle =0 \}.
\end{align*}
Then $\Ten{P}{\Upsilon}$ is a subspace, $\Id{S}{\Upsilon}$ is an ideal, and
$\Op{S}{P}$ is an algebraic set, and they satisfy the following Galois
correspondence property:
\begin{align}\label{eqn:Galois}
   S \subset \Ten{P}{\Upsilon} \quad \Longleftrightarrow \quad
   P \subset \Id{S}{\Upsilon} \quad \Longleftrightarrow \quad
   \Upsilon \subset \Op{S}{P}.
\end{align}

For an algebraic perspective, each $\omega\in\Omega$ defines a
representation $\rho_{\omega}\colon K[X]\to \End(T)$, where $f\mapsto (\bra{t}\mapsto
\bra{t} f(\omega))$. The sets $\Id{S}{\Upsilon}=\bigcap_{\omega\in
\Upsilon}\Id{S}{\omega}$ are annihilator ideals in $K[X]$, so they are a
multilinear generalization of the concept of minimal polynomials.  The sets $\Ten{P}{\Upsilon}$ generalize tensor products. For
example, if $A\subseteq \End(V_1)\times \End(V_2)^{{\rm op}}$ is an
associative algebra and $f(X)=x_1-x_2$, then 
$\Ten{x_1-x_2}{A}$ is the usual tensor product 
$(V_1\otimes_A V_2)^*$. For, if $t\in (V_1\otimes_A V_2)^*$ 
and all $(\phi_1,\phi_2)\in A$, then
\begin{align*}
   \langle t|f(\phi_1,\phi_2)\ket{v_1,v_2} = \langle t| \phi_1v_1, v_2 \rangle - \langle t | v_1, \phi_2v_2\rangle = 0.  
\end{align*}
The algebraic sets $\Op{S}{P}$ may, depending on $P$, come
equipped with algebraic structure external to their definition. For example, 
\begin{align*}
   \Adj(S) & = \Op{S}{x_1-x_2} & 
   \Der(S) &= \Op{S}{x_1+\cdots+x_{\vav}}
\end{align*}
are, respectively, associative and Lie algebras.

The Galois correspondence in~\eqref{eqn:Galois} relating 
these three constructions has closures. For example, if $\textbf{d} =
x_1+\cdots+x_{\vav}$, then
\begin{align*}
   \den{t} &= \Ten{\textbf{d}}{\Der(t)} = \Ten{\textbf{d}}{\Op{t}{\textbf{d}}}.
\end{align*}
The adjoint-tensor uses a different closure: 
\begin{align*} 
   (V_1\otimes_{\Adj(t)} V_2\otimes V_0^*)^* = \Ten{x_1-x_2}{\Op{t}{x_1-x_2}}. 
\end{align*} 
In fact, the following proposition elucidates a tensor compression method for {\em every} 
ideal in $K[X]$! Note, $\Omega^{\times}=\GL(V_1)\times 
\cdots \times \GL(V_{\vav})$ is the group 
of units of $\Omega$.

\begin{prop} \label{prop:general-iso}
   Let $s,t\in (V_1\otimes \cdots \otimes V_{\vav})^*$
   and let $P$ be an ideal of
   $K[X]$. Then there exists $\phi\in\Omega^{\times}$
    such that $s^\phi = t$ if, and only if, there
   exist $\mu,\nu\in \Omega$ such that 
   \begin{align*} 
      \Op{s}{P}^{\mu} & = \Op{t}{P}, \\
      \Op{t}{P}^{\nu} & = \Op{t}{P}, \text{ and}\\
      (s^\mu)^\nu & = t\in \Ten{P}{\Op{t}{P}}.
   \end{align*} 
\end{prop}

\begin{proof} 
   Suppose $\phi\in\Omega^{\times}$ satisfies 
   $s^{\phi}=t$. Fix
   $\omega\in \Op{s^{\phi}}{P}$ and
   $f=\sum_e\lambda_e X^e\in P$. Then, for all $v_a\in V_a$,
   \begin{align*} 
      0 &= \langle s^{\phi} | f(\omega) | v_1 ,\dots, v_{\vav}\rangle = \sum_e \lambda_e \langle s^{\phi} | \omega_1^{e_1} v_1, \dots, \omega_{\vav}^{e_{\vav}}v_{\vav} \rangle \\
      &= \sum_e \lambda_e \langle s | \phi_1\omega_1^{e_1}\phi_1^{-1} u_1, \dots, \phi_{\vav}\omega_{\vav}^{e_{\vav}}\phi_{\vav}^{-1}u_{\vav} \rangle ,
   \end{align*}
   where $u_a = \phi_a v_a$. It follows that
   $\Op{s}{P}^{\phi^{-1}} = \Op{s^{\phi}}{P} = \Op{t}{P}$.
   Now put $\mu=\phi^{-1}$ and $\nu=1$. The reverse 
   implication is proved similarly.
\end{proof}

\subsection{The derivation-densor method}
\label{sec:der-den}
There are many possible ideals, $P$, one can consider 
to seed the mechanism in Proposition~\ref{prop:general-iso}.  
To narrow the candidate pool we impose three requirements:
(a) $\Op{t}{P}$ has an algebraic structure like 
$\Adj(t)$ and $\Der(t)$; (b) the choice of 
$P$ is independent of the given tensor $t$; 
and (c) there is an efficient algorithm to construct 
$\Op{t}{P}$.

For several reasons,  the ideal $P=(\mathbf{d})$ with $\mathbf{d} = x_1 + \cdots +x_{\vav}$ is the perfect candidate. 
First, $\Der(t)=\Op{t}{(\mathbf{d})}$ is the solution space 
of a system of linear equations, and hence can be constructed efficiently.
Secondly, all associative algebras associated to 
$t$, such as $\Adj(t)$, embed in 
$\Der(t)$~\cite{BMW:exact}*{Theorem~A}.
Further, for any ideal $P$ generated by linear homogeneous polynomials,
the densor subspace $\den{t}$ embeds in 
$\Ten{P}{\Op{t}{P}}$~\cite{FMW:densors}*{Theorem~A}.
{\em
Thus, in a precise sense, $\den{t}$ is the most compressed space one 
can use with linear methods.
}

Since $\Der(t)=\Op{t}{(\mathbf{d})}$ and
    $\den{t}=\Ten{(\mathbf{d})}{\Op{t}{(\mathbf{d})}}$,
the correctness of the derivation--densor method 
follows directly from   
Proposition~\ref{prop:general-iso} 
applied to $P=(\mathbf{d})$:

\begin{thm}\label{thm:correctness}
	Algorithm~\ref{algo:der-densor} decides isomorphism of tensors $s,t\in
	(V_1\otimes\cdots\otimes V_{\vav})^*$.
\end{thm}

\section{Deciding pseudo-isomorphism of simple Lie modules}
\label{sec:primitive}

We turn now to the key steps in 
Algorithm~\ref{algo:der-densor},
namely to the related tasks in 
Lines 2 and 4. As noted in Section~\ref{sec:intro-use},
it is convenient to translate those tasks into 
pseudo-isomorphism problems for modules.
The purpose of this section is to 
solve the latter module problems 
for the classes of algebras we consider in this article.
In particular, we will prove 
Theorem~\ref{thm:Lie-conjugacy}.

\subsection{Three illustrations}
\label{sec:three-exs}
In associative and Lie settings,
pseudo-isomorphism of modules is a strictly coarser 
equivalence than isomorphism.
An analysis of the associative setting is provided in~\cite{BW:mod-iso}. 
To elucidate the differences between module isomorphism and
pseudo-isomorphism, and to distinguish the Lie module setting from its
associative counterpart, we briefly describe three computational settings. 
Some challenging obstructions 
are encountered even in these elementary examples. 

\subsubsection{Irreducible representations of simple algebras}
\label{sec:simple-simple}
Consider $L=\mathfrak{sl}_3(K)$ acting in two different ways on $K^3$ as
follows: for $x\in L$ and $v\in K^3$,
\begin{align}
   [x,v]_1 & = xv & [x,v]_2 = \bar{x}v,
\end{align}
where $\bar{x}_{ij}=x_{(3-j+1)(3-i+1)}$ is the transpose along the opposite
diagonal.  Isomorphism can be decided using the standard theory of weight spaces,
as described in \cite{Humphreys}*{Chapter VI}.  The highest weight space of the first
module, ${_L V_1}$, is $V_{\lambda}=Ke_1$, where $\lambda$ has support
$h_1=E_{11}-E_{22}$ in the standard Cartan subalgebra. The highest weight space
of the second module, ${_L V_2}$, is the same space but with a different weight,
namely $V_{\lambda'}=Ke_1$ where $\lambda'$ has support $h_2=E_{22}-E_{33}$.
Thus, ${_L V_1}$ and ${_L V_2}$ are non-isomorphic
$L$-modules~\cite{Humphreys}*{VI.20.3}. However, $\Phi:=I_3$ and $\phi\colon
x\to\bar{x}$ is evidently a pseudo-isomorphism ${_L V_1}\to {_L V_2}$,  
which could be termed a ``graph-twist''
because $\phi$
induces an automorphism of the Dynkin diagram of $L$.  

J. Grochow observed in his Ph.D. thesis that, 
when $L$ is a simple Chevalley Lie algebra,
the isomorphism classes of simple $L$-modules are determined by
such graph automorphisms~\cite{Grochow:thesis}. 
In that case, one can exhaust the limited
number ($\leq 6$) of Dynkin diagram automorphisms until an appropriate 
choice for $\phi\colon L_1\to L_2$ is found to reduce the given pseudo-isomorphism problem to an instance of isomorphism that may be solved by the theory of weight spaces. 
For this observation to be practicable, one requires
efficient algorithms to recognize that a given Lie algebra has a Chevalley basis
and to construct one if it does. Fortunately, such algorithms
exist~\citelist{\cite{dG:book}\cite{MW:Chevalley}\cite{Ryba}}, so
Theorem~\ref{thm:Lie-conjugacy} holds when the given Lie algebras are simple.

These sorts of pseudo-isomorphisms between modules of simple Lie algebras
have no associative analogue: by the Skolem--Noether theorem, every automorphism
of a simple associative algebra is inner.

\subsubsection{Completely reducible representations of semisimple algebras}
\label{sssec:sem-sem}
Let $L=\mathfrak{sl}_d(K)^n$, and define two actions 
on $V^n=K^{dn}$ by
\begin{equation}
   \begin{split}
      [(x_1,\ldots,x_n),(v_1,\ldots,v_n)]_1 & = (x_1v_1,\ldots,x_nv_n) \\
      [(x_1,\ldots,x_n),(v_1,\ldots,v_n)]_2 & = (x_{\sigma (1)}v_1,\ldots,x_{\sigma (n)}v_{n}).         
   \end{split}
\end{equation}
where $\sigma$ is a permutation of $[n]:=\{1,\dots, n\}$.  
For large values of $n$ it would be prohibitively expensive to list
all permutations as we did with the Dynkin diagram automorphisms.  
However, more thoughtful strategies also have their limitations: 
Grochow proved that 
pseudo-isomorphism in this setting is at least as 
hard as the Graph Isomorphism 
problem~\cite{Grochow:thesis}.  

There is an analogous situation for associative
algebras---where it is equally futile to list all possible permutations---but an
efficient solution exists in this case~\cite{BW:mod-iso}*{Theorem~1.3}.  

\subsubsection{Irreducible representations of semisimple algebras}
\label{subsec:tensor}
For $i\in[m]$, let $E_i$ be a field extension of $K$.
Consider
$L = \mathfrak{sl}_{d_1}(E_1)\oplus \cdots \oplus \mathfrak{sl}_{d_m}(E_{m})$
acting on $V=E_1^{d_1}\otimes \cdots \otimes E_{\ell}^{d_{m}}$ by 
\begin{align}
   [(x_1,\ldots,x_{\ell}),v_1\otimes\cdots\otimes v_{m}]
   & = x_1v_1\otimes\cdots\otimes x_{\ell}v_{m},
\end{align}
where $x_i\in \mathfrak{sl}_{d_i}(K_i)$, $v_i\in E_i^{d_i}$.  Again, permutations of coordinates threatens to encode
a hard problem as an instance of a pseudo-isomorphism problem of this type.  
However, unlike Section~\ref{sssec:sem-sem}, we have a tensor product rather
than a direct sum, so minimal ideals of $L$ do not annihilate subspaces of the
module.  

This situation is again particular to the nonassociative setting. 

\subsection{Tensor decompositions of simple Lie modules}
\label{sec:mod-equiv}
The situation described in Section~\ref{subsec:tensor} illustrates a general
phenomenon. If $L=M\oplus N$ is a Lie algebra decomposition into ideals, and $V$
is an $L$-module, then a consequence of the Jacobi identity is that $L=M\oplus
N$ acts on $V$ in the following sense:
\begin{align}
   \label{def:transverse}
      &(\forall m\in M)(\forall n\in N)(\forall v\in V)& m(nv)& =n(mv).
\end{align}
This property enables us to characterize, constructively
 as iterated tensor products, the Lie modules
arising in Theorem~\ref{thm:Lie-conjugacy}.

For an ideal $M$ of $L\leq\mathfrak{gl}(V)$, let $K\langle M\rangle\leq \End(V)$
denote its associative envelope, 
namely the $K$-span of the semigroup generated
by $M\subseteq\End(V)$. For $S\subset V$, put $MS:=K\langle M\rangle S$, the
smallest $M$-submodule of $V$ containing $S$. The following elementary result
provides the engine for our decomposition algorithm.

\begin{lemma}
\label{lem:split-SS}
   Let $L=M\oplus N$ be a nontrivial decomposition with 
   $M$ a minimal ideal, and let $V$ be a  
   simple $L$-module. If $S$ is a
   proper, simple $M$-submodule of $V$, then $V=S\oplus NS$ is an $M$-module
   decomposition, and $S$ embeds in $NS$.
\end{lemma}

\begin{proof}
   For an $M$-submodule, $S$, of $V$, we have $M(NS)=N(MS)\leq NS$ by~\eqref{def:transverse}, so $NS$ is an $M$-submodule. As
   $V$ is a simple $L$-module, $V=LS=MS+NS\leq S+NS$. Since $S$ is a proper,
   simple $M$-module, $S\cap NS=0$ as required. Since $S$ is not an
   $L$-submodule of $V$, it follows that $NS\neq 0$, so there exists $n\in N$
   with $nS\neq 0$. By~\eqref{def:transverse} again, $s\mapsto ns$ is an
   $M$-module embedding $S\to NS$.
\end{proof}

To ensure linear algebra is done over fields rather than division rings, 
we introduce an additional condition on minimal ideals $M$ of a Lie algebra 
$L$. 
An ideal $M$ of $L$ 
has \emph{central type} if $K\langle
M\rangle$ is isomorphic to $\mathbb{M}_f(\Delta)$ for some
integer $f$ and field $\Delta$. We say $L$ has \emph{central type} if all
minimal ideals of $L$ have central type. 

\begin{lem}\label{lem:central-type}
   If $M$ is a simple Lie algebra of central type and $S$ a simple faithful
   $M$-module, then $\End_M(S)$ is a field. 
\end{lem}

\begin{proof}
   By Schur's lemma, $\Delta=\End_M(S)$ is a division ring, and by
   Wedderburn--Artin, $M\cong \End_{\Delta}(S)$. Since $M$ is of central type,
   $\Delta$ is a field. 
\end{proof}

Our decomposition algorithm also makes use of idempotents of a matrix algebra
over $K$. An element $e\in A\leq \End(V)$ is a \emph{idempotent} if
$e^2=e$. Two idempotents $e,f\in A$ are \emph{orthogonal} if $ef=0=fe$. Finally,
an idempotent $e\in A$ is \emph{primitive} if $e$ is not the sum of two nonzero
orthogonal idempotents.

\begin{thm}
\label{thm:tensor-decomp}
   There is a polynomial-time algorithm that, given a decomposition $L=M\oplus
   N$, with $M$ minimal and of central type, and a simple $L$-module
   $V$, returns an \(M\)-submodule $S\leq V$, an \(N\)-submodule $T\leq V$, and
   an $L$-module isomorphism \(V\to S\otimes_{\End_M(S)} T\).
\end{thm}

\begin{proof}
   There are two stages to the algorithm. The first constructs a decomposition
   $V\cong S\otimes_{\Delta} \Delta^r$ for a simple $M$-submodule $S$, where
   $\Delta = \End_M(S)$. The second constructs an $N$-submodule $T\leq V$ and an
   isomorphism $T\to \Delta^r$. 
      
   We first verify, in polynomial time, that $M$ is minimal by checking 
   that $K\langle M\rangle$ is simple using~\cite{Ronyai}. Then
   we apply polynomial-time algorithms from~\citelist{\cite{Ronyai}
   \cite{HR:irreducibility}} to test whether $V$ is a simple $M$-module. If $V$
   is simple, set $S=V$. Otherwise the algorithms of~\citelist{\cite{Ronyai}
   \cite{HR:irreducibility}} produce a proper, simple $M$-submodule $S\leq V$. 
   
   Initialize $\mathcal{V}:=\{S\}$. While $\sum_{U\in \mathcal{V}}U\neq V$, find
   a generator $n\in N$ such that $nS\cap \sum_{U\in\mathcal{V}}U=0$, and put
   $\mathcal{V}:=\mathcal{V}\cup\{nS\}$. By the proof of
   Lemma~\ref{lem:split-SS}, $S$ and $nS$ are isomorphic $M$-submodules of $V$,
   so we obtain an isomorphism $\alpha \colon V\to S^{\oplus r}$ of $M$-modules.
   By solving a system of linear equations, we construct generators for the
   field $\Delta=\End_M(S)$, cf.~Lemma~\ref{lem:central-type}. We use $\alpha$
   to write $V\cong S\otimes_{\Delta} \Delta^r$. This completes the first stage
   of the algorithm.

   We begin the second stage by computing generators for $A:=K\langle
   M\rangle\subset \End_K(V)$. We next use the recognition algorithm
   of~\cite{IRS}*{Theorem~1} to construct a primitive idempotent $e\in A$,
   and put $T:= eV$. Since $e$ is a primitive idempotent of $A\cong
   \mathbb{M}_f(\Delta)$, there is a natural ring isomorphism $eAe\to \Delta$,
   and a $\Delta$-vector space isomorphism $\Delta \to eS$. Hence,
   \begin{align*}
      S\otimes_{\Delta} T
         & \cong S\otimes_{\Delta} (eS\otimes_{\Delta} \Delta^r)\\
         & \cong S \otimes_{\Delta} (\Delta \otimes_{\Delta} \Delta^r)\\
         & \cong S\otimes_{\Delta} \Delta^r \cong V. \qedhere
   \end{align*}
   The relevant mappings in the above isomorphisms are a by-product of the 
   computation as well.
\end{proof}

\begin{cor}
\label{cor:transverse}
   Let $L=M_{1}\oplus \cdots\oplus M_{r}$ be a decomposition into nontrivial
   minimal ideals of central type, and let $V$ be a simple $L$-module. Then $V\cong
   S_1\otimes\cdots\otimes S_r$, where $S_i\leq V$ is a simple $M_i$-submodule
   for $i\in [r]$. Furthermore, if $L$ is of central type and the decomposition
   of $L$ into ideals is given, then the tensor decomposition of $V$ can be
   constructed in polynomial time.
\end{cor}

One final ingredient is needed for the proof of 
Theorem~\ref{thm:Lie-conjugacy}. 
For finite fields one can, in
polynomial time, test whether a given (unital) associative 
algebra is cyclic (a quotient of a polynomial ring), and
decide pseudo-isomorphism of modules for such algebras~\cite{BW:mod-iso}. 
Those algorithms have since been generalized to number
fields~\cite{IRS}*{pp.~211--212}; the next result follows
\emph{mutatis mutandis} from~\cite{BW:mod-iso}*{Theorem~1.3}.

\begin{thm}
\label{thm:mod-iso}
   Fix a field $K$ that is finite or finite over $\mathbb{Q}$, and a
   finite-dimensional vector space $V$. There is a polynomial-time algorithm to
   decide if an associative algebra $A\leq \End_K(V)$ is cyclic and another to
   settle pseudo-isomorphism for modules of such algebras.
\end{thm}

\subsection{Proof of Theorem~\ref{thm:Lie-conjugacy}} 
Let $L_1$ and $L_2$ be the given Lie algebra of Chevalley type represented
faithfully on simple modules $V_1$ and $V_2$, respectively. 
For $i=1,2$, $L_i$ is reductive (it has central nil radical),
so it has central type and
decomposes as $L_i=M_{i0}\oplus M_{i1}\oplus \cdots \oplus M_{ir_i}$ 
into nontrivial minimal ideals, with $M_{i0}$ abelian. 
Moreover, such a decomposition can be found
using~\cite{IRS}*{Theorem~1} and the more general finite 
field case discussed in \cite{IRS}*{pp.~211--212}. We may assume $r_1=r_2=r$ since, otherwise,
$L_1\not\cong L_2$. By re-indexing we may further assume, for each $k\in [r]$,
that $M_{1k}\cong M_{2k}$  as Lie algebras by computing Chevalley
bases---possibly over extension fields---and comparing root data. For the
abelian ideals $M_{10}$ and $M_{20}$, we simply compare dimensions.

We first handle the abelian ideals. By Schur's Lemma, $K\la M_{10}\ra$ and $K\la
M_{20}\ra$ are both cyclic algebras. Using Theorem~\ref{thm:mod-iso}, we
construct $\psi_0\colon K\la M_{10}\ra\rightarrow K\la M_{20}\ra$ and
$\Psi_0\colon V_1\to V_2$ such that $(\Psi_0,\psi_0\oplus \mathrm{id}_{1}
\oplus\cdots \oplus \mathrm{id}_{r})$ is a pseudo-isomorphism ${_{K\la
M_{10}\ra} V_1}\to {_{K\la M_{20}\ra} V_2}$.

Next, for each $i\in [2]$, apply Corollary~\ref{cor:transverse} to construct a
tensor decomposition $V_i =S_{i1}\otimes \cdots \otimes S_{ir}$, where $S_{ij}$
a simple $M_{ij}$-module for $j\in [r]$.  For each $j\in [r]$, use Grochow's
algorithm \cite{Grochow:thesis} (discussed in Section~\ref{sec:simple-simple})
to construct a pseudo-isomorphism $(\Psi_j,\psi_j)$ from ${_{M_{1j}} S_{1j}}\to
{_{M_{2j}} S_{2j}}$. If the latter fails for some $j$, then there is no
pseudo-isomorphism ${_{L_1}V_1}\to {_{L_2} V_2}$, so we report that and exit.
Otherwise, $((\Psi_1\otimes\cdots\otimes \Psi_r)\cdot \Psi_0,\psi_0\oplus
\psi_1\cdots\oplus \psi_r )$ is the desired pseudo-isomorphism ${_{L_1}V_1}\to
{_{L_2} V_2}$.  \hfill $\Box$

\section{Families of densors with prescribed dimensions}
\label{sec:examples}

In this section, we construct an infinite family 
of tensors with small densor spaces. In 
particular, there is a sub-family whose 
densor spaces are $1$-dimensional. These tensors come from the
classical representation theory of $\sl_n$-modules,
and we expect that similar
ideas can be used to build more such families. 

Throughout this section, $K$ will  
denote a field that is either finite
or finite over $\mathbb{Q}$. Let
$n$ be a positive integer such that if ${\rm char}(K)=p>0$ 
then $p\nmid(n+1)$. 
Let $L=\mathfrak{sl}_{n+1}(K)$, the simple Lie
algebra of type $A_n$, and let $M$ be a 
finite-dimensional simple $L$-module. The Lie
module operation is a $K$-bilinear map, 
$\bra{t}\colon L \times M \rightarrowtail M$, 
and $\delta=(\delta_2,\delta_1,\delta_0)$ is 
a \emph{derivation} of $t$ if, for all $x\in L$ and $v\in M$, 
\begin{align}\label{eqn:bimap-der}
    \delta_0\langle t | x, v\rangle = \langle t | \delta_1x, v\rangle + \langle t | x, \delta_2v\rangle. 
\end{align}
Equivalently, we could construct from $t$
a trilinear form, giving rise to a
natural bijection between the derivations in~\eqref{eqn:bimap-der} and the
derivations defined in Section~\ref{sec:intro}.

\begin{lem}\label{lem:simple-subalgebra}
    $\Der(t)$ contains a simple subalgebra isomorphic to $\mathfrak{sl}_{n+1}(K)$. 
\end{lem}

\begin{proof}
    Since $M$ is an $L$-module, it follows that for all $v\in M$ and $x,y\in L$,
    \begin{align*} 
        \langle t | xy, v \rangle &= (xy)v = x(yv) - y(xv) = x \langle t | y, v\rangle - \langle t | y, xv\rangle. 
    \end{align*} 
    Therefore, $L$ embeds into $\Der(t)$, and the lemma follows since $L\cong
    \mathfrak{sl}_{n+1}(K)$. 
\end{proof}

The simple $L$-module $M$ contains a unique vector of highest weight $\lambda$.
We write $M = V(\lambda)$ if $M$ is an $L$-module with highest weight $\lambda$,
where $\lambda$ is a partition with $n$ parts, possibly equal to $0$. Write
$\lambda = (\lambda_1,\dots, \lambda_n)\vdash m$ if $\sum_i\lambda_i = m$. We
need to determine the number of irreducible submodules of $V(\lambda)\otimes
V(\mu)$ isomorphic to $V(\nu)$, which are the Littlewood--Richardson numbers for
type $A$, denoted by $c_{\lambda, \mu}^\nu$. These numbers can be computed by
algorithms on Young tableaux, similar to the well-known $\mathfrak{gl}_n$ case.
We follow closely the notation used in~\cite{HK:Crystals}. 

We denote by $Y$ a Young diagram of type $\lambda \vdash m$. Let
$\mathcal{B}(Y)$ be the set of \emph{semi-standard Young tableaux} obtained by
filling in the boxes of the diagram $Y$ with integers $[n+1]$ such that each row
is weakly increasing and each column is strictly increasing. A tableau is
\emph{standard} if the integers $1$ through $m$ appear once. 

For a Young diagram $Y$ of type $\lambda = (\lambda_1,\dots, \lambda_n)$, define
a new Young diagram
\begin{align}\label{eqn:row-insert}
    Y[j] = \left\{\begin{array}{ll}
        (\lambda_1, \dots, \lambda_j+1, \dots, \lambda_n) & j\leq n, \\
        (\lambda_1 - 1, \dots, \lambda_n - 1) & j = n+1.
    \end{array}\right. 
\end{align}
For $m\geq 2$, we define $Y[b_1,\dots, b_{m-1}, b_m]$ recursively so that 
\begin{align*}
    Y[b_1,\dots, b_{m-1}, b_m] = Y[b_1,\dots, b_{m-1}][b_m],
\end{align*}    
provided $Y[b_1, \dots, b_i]$ is a Young diagram for all $i\in [m-1]$. The next
theorem states how this operation can be used to determine $c_{\lambda,
\mu}^\nu$. 

For a Young diagram of type $\lambda$ with $n$ parts, (a basis of) the weight
space decomposition of $V(\lambda)$ corresponds to the set $\mathcal{B}(Y)$. We
abuse notation and identify the two, working with tableaux instead. So we write
$\mathcal{B}(Y)\oplus\mathcal{B}(Y')$, the disjoint union of tableaux, for the
module $V(\lambda)\oplus V(\lambda')$. 

\begin{thm}[{\cite{HK:Crystals}*{Theorem~8.6.6}}]\label{thm:sl_n-crystals}
    Let $\lambda$ and $\mu$ be partitions with $n$ parts, and let $Y$ and $Y'$
    be the corresponding Young diagrams. Then there exists an isomorphism of
    $\mathfrak{sl}_{n+1}$-modules
    \begin{align*}
        \mathcal{B}(Y) \otimes \mathcal{B}(Y') \cong 
        \bigoplus_{b_1\otimes \cdots \otimes b_m \in\mathcal{B}(Y')} \mathcal{B}(Y[b_1,\dots, b_m]). 
    \end{align*}
    Note, if $Y[b_1,\dots, b_m]$ is not a Young diagram, then
    $\mathcal{B}(Y[b_1,\dots, b_m]) = 0$. 
\end{thm}

\begin{prop}\label{prop:LR-combinatorics}
    With $n\geq 1$, set $\mu = (2, 1, \dots, 1)\vdash n+1$. If $\lambda =
    (\lambda_1, \dots, \lambda_n)$ is a partition, then $c_{\lambda,
    \mu}^\lambda = |\{\lambda_i ~|~ 1\leq i\leq n,\; \lambda_i > 0 \}|$. 
\end{prop}

\begin{proof}
    Write $\lambda = (n_1,\dots, n_1, n_2,\dots, n_2,\dots,
    n_k,\dots, n_k)$ for some $k\leq n$, where $n_i > n_{i+1}$ for
    $i\in [k-1]$. Let $Y$ and $Y'$ be the Young diagrams corresponding 
    to $\lambda$ and $\mu$ respectively. We count
    the number of summands equal to $\mathcal{B}(Y)$ in
    $\mathcal{B}(Y)\otimes\mathcal{B}(Y')$. From 
    Theorem~\ref{thm:sl_n-crystals} these correspond to
    tableaux $T:=b_1\otimes \cdots \otimes
    b_{n+1}\in \mathcal{B}(Y')$ such that $Y[b_1,\dots, b_{n+1}]=Y$.
    The latter condition implies that $T$ is a standard Young
    tableau of type $\mu=(2, 1, \dots, 1)$, so $b_2 = 1\ne b_1$.

    If $n_k = 0$, since $b_1\ne n+1$
    there are $k-1$ choices for $b_1$ such that $Y[b_1]$ is a Young
    tableau. If $n_k>0$, there are $k$ choices
    for $b_1$. Since $1 = b_2 < b_3 < \cdots < b_{n+1}\leq n+1$, the remaining
    $b_i$ in both cases are uniquely determined. 
    Thus, $c_{\lambda, \mu}^\lambda \in
    \{k-1, k\}$ depending only on whether $n_k=0$ or $n_k>0$;
    the result follows.
\end{proof}

Using these results we can now build families 
of tensors with 1-dimensional densors.
If $m$ is another positive integer, let
$d(m, n)$ 
be the number of divisors of $m$ no larger than $n$.

\begin{thm}\label{thm:infinite-family}
    For any $K$ there are infinitely many positive integers $n$ such that, for
    all positive integers $m$, there are at least $d(m, n)$ pairwise
    non-isomorphic $K$-tensors with $1$-dimensional densor space.
\end{thm}

\begin{proof}
    Set $L=\mathfrak{sl}_{n+1}(K)$. By Lemma~\ref{lem:simple-subalgebra},
    $\Der(t)$ contains a simple subalgebra $D\leq\Der(t)$ isomorphic to $L$.
    Setting $\textbf{d} = x_2 + x_1 - x_0$, we consider only
    $\Ten{\textbf{d}}{D}$ in place of $\den{t}$. Note that $\den{t} \leq
    \Ten{\textbf{d}}{D}$. We will show that $\dim\Ten{\textbf{d}}{D} = 1$, so
    that $\den{t} = \Ten{\textbf{d}}{D}$ as $0\neq t\in\den{t}$. 

    Since $M$ and $L$ are irreducible $L$-modules, they are irreducible
    $D$-modules. Every tensor contained in $\den{t}$ determines a
    $\Der(t)$-module homomorphism $M\rightarrow M\otimes L$, which must also be
    a $D$-module homomorphism. Each irreducible $L$-module has a unique vector
    of highest weight, so there exist partitions $\lambda$ and $\mu$ such that
    $M\cong V(\lambda)$ and $L\cong V(\mu)$ as $D$-modules. Since $L$ is the
    adjoint module, $\mu = (2, 1,\dots, 1)\vdash n+1$. By irreducibility, the
    number of $K$-linearly independent $D$-module homomorphisms $M\rightarrow
    M\otimes L$ is equal to the generalized Littlewood--Richardson number for
    type $A$, namely $c_{\lambda, \mu}^\lambda$. For $m\geq 1$ and for all
    positive integers $\ell$ such that $\ell\mid m$ and $\ell\leq n$, let
    $\lambda \vdash m$ with parts of size $\ell$ and $0$. From
    Proposition~\ref{prop:LR-combinatorics}, $c_{\lambda, \mu}^\lambda = 1$.
    There are at least $d(m, n)$ such partitions $\lambda$, which proves the
    theorem.
\end{proof}

\section{Proof of Theorem~\ref{thm:tensor-iso}}
\label{sec:tiny}

Let $t\in(V_1\otimes\cdots\otimes V_{\vav})^*$ be 
nondegenerate, and assume $L:=\Der(t)$ 
is reductive. 
\medskip

First we show that if some $V_a$ is non-simple as an $L$-module
then $\dim\den{t}>1$.
Suppose, for some 
$a\in[\ell]$, that $U_a$ is a proper nontrivial 
$L$-submodule of $V_a$. Let $e\colon V_a\to V_a$ be an 
idempotent with kernel $U_a$, and set 
$\langle s| v\rangle=\bra{t} v_1, \dots,  ev_a, \dots, v_\ell\rangle$.
Since $L$ is reductive, the image of $e$ is an $L$-module complement 
to $U_a$. Thus, for each $\delta\in L$, 
$e\delta_a = \delta_a e$. It follows that 
$L \subseteq \Der(s)$, and hence that $s\in \den{t}$. 
Because $U_a$ is nontrivial and proper, $s$ is nonzero 
and degenerate. Since $t$ is nondegenerate, 
$s$ and $t$ are linearly independent vectors 
so $\dim \den{t}>1$. 
\medskip

Next, let $t_1,t_2\in (V_1\otimes \cdots \otimes V_{\vav})^*$ 
be two nondegenerate tensors having 1-dimensional 
densor spaces.
We apply 
Algorithm~\ref{algo:der-densor} to test for isomorphism.

First, the derivation algebras $L_i:=\Der(t_i)$ (Line 1) are constructed in
polynomial time by solving a linear system. 
By assumption, each $L_i$ is reductive which allows us to 
decompose $L_i=M_{i1}\oplus \cdots \oplus M_{ir_i}$ 
into nontrivial minimal ideals (see \cite{IRS}*{Theorem~1} and the
more general finite field case discussed in \cite{IRS}*{pp.~211--212}). 
If $r_1\neq r_2$, then $\Der(t_1)$ is not conjugate to 
$\Der(t_2)$.
   
As $\dim \den{t_i}=1$, for each $a\in [\vav]$, $V_a$ is a simple 
$L_i$-module ($i=1,2$).  So we may apply 
Theorem~\ref{thm:Lie-conjugacy} to construct 
$\phi_a\in \GL(V_a)$ such that
$(L_1|_{V_a})^{\varphi_a}=L_2|_{V_a}$.  
The action of $L_i=M_{i1}\oplus \cdots \oplus M_{ir_i}$ on
$V_1\otimes\cdots \otimes V_{\vav}$ satisfies the property in~\eqref{def:transverse}, 
so setting $\varphi:=\varphi_1\otimes\cdots\otimes 
\varphi_{\vav}$, gives $L_1^{\varphi}=L_2$ in 
$\End( (V_1\otimes\cdots \otimes V_{\vav})^*)$. 
This completes Line 2 of Algorithm~\ref{algo:der-densor}. 

Since $\dim\den{t_i}=1$, we do not need to induce images 
of normalizers.  Therefore, we proceed to Line~4, 
where the task is merely to decide if $t_1^{\varphi}=\lambda t_2$ 
for some scalar $\lambda$. This is settled by 
solving a tiny linear equation, so the result follows. \hfill\qed

\begin{remark}
\label{rem:fast}
    Theorem~\ref{thm:tensor-iso} decides isomorphism 
    within the family of
    tensors in Theorem~\ref{thm:infinite-family} 
    in polynomial time, but we are aware of no
    other sub-exponential isomorphism tests for this 
    family. For instance, 
    if $t$ is a tensor in this family, a
    consequence of the construction is that 
    $\Adj(t) \cong K$. Thus, the
    adjoint-tensor method is no better than 
    brute force for this family of tensors. 
\end{remark}

\section{Further results}
\label{sec:generalizing}
There are a number of similar results attainable by modest adaptation of our
methods.  We are careful to avoid constructing $N(\Der(t))$ for general 
fields $K$, since $K^{\times}$ may not have a finite generating set.  
When $K$ is finite, however, we can give generators for $K^{\times}$ and,
consequently, also for $N(\Der(t))$. 

\begin{thm}\label{thm:auto}
    Let $K$ be a finite field with $K=6K$, 
    and let $t\in (K^{d_1}\otimes \cdots \otimes
    K^{d_{\vav}})^*$ satisfy the hypotheses of 
    Theorem~\ref{thm:tensor-iso}. In polynomial time 
    one can construct generators for the group $\Aut(t)$.
\end{thm}

When $\dim\den{t}>1$ it is still possible that  
$\Der(t)$ is reductive and irreducible on each $V_a$.  
In that case we are left to search the orbit of 
$N(\Der(t))$ acting on $\den{t}$. 
\medskip

When $\Der(t)$ is represented {\em reducibly} on the $V_a$, one is
confronted with familiar difficulties when matching simple factors. 
Indeed, Grochow has
shown that a general solution to the conjugacy problem for semisimple Lie
algebras over any field requires solving Graph
Isomorphism~\cite{Grochow:thesis}. However,
this obstruction is not so formidable when the number 
of simple $\Der(t)$-modules is bounded. 
\medskip

The situation when $\Der(t)$ has a noncentral nil radical is
worse. Indeed, the existence of radicals is a 
problem even for associative algebras~\cite{BW:tensor}.
Although the presence of a flag suggests that an inductive process may succeed,
all actions must also normalize the radical.  This extra condition is 
itself a tensor isomorphism problem, but now involving 
tensors that 
arise as the product of the nil radical.  It is not known if this case is 
as hard as the general case of tensor isomorphism, but 
certainly no efficient solution is known.
\medskip

Although the derivation algebras of tensors over fields of positive
characteristic are restricted Lie algebras, they can have (nonabelian) simple
factors that are not of Chevalley type. For example, let $A=\F_p[x]/(x^p)$,
for a prime $p$, and define $\bra{t} \colon A^2\times A^2\rightarrowtail A$ via
\begin{align*} 
    \bra{t} (a, b), (x, y)\rangle &= ay - bx.
\end{align*} 
This tensor can also be interpreted as the commutator of the Heisenberg group
$H(A)$. The derivation algebra of $t$ is isomorphic to $\Der(A)\oplus
\sl_2(A)\oplus A^2$, where $\Der(A)$ is the simple
$p$-dimensional Jacobson--Witt Lie algebra of derivations of $A$. Over
$\mathbb{F}_p$, the tensor appears to have a $1$-dimensional densor subspace
for some small primes. By Corollary~\ref{cor:transverse}, 
we can extend
Theorem~\ref{thm:tensor-iso} to a broader class, $\mathcal{C}$, 
of restricted Lie algebras, provided we have a polynomial-time algorithms to decide
pseudo-isomorphism of simple modules over simple Lie algebras in
$\mathcal{C}$. 
\medskip

We have implemented prototypes of our algorithms in 
the {\sc Magma} system~\cite{Magma}. They are publicly available within software packages for
effective computation with tensors~\cite{TensorSpace}.

\subsection*{Acknowledgements}
We thank W.A.~de Graaf and J.~Grochow for answers to questions about the
conjugacy of Lie matrix algebras.  We also thank the Hausdorff Institute
for Mathematics trimester on \emph{Logic and Algorithms in Group Theory} 
and The Newton Institute special program on \emph{Groups, representations
and applications} where some of this research was undertaken. Finally, we 
thank the anonymous referee for suggesting improvements to the 
exposition.

\end{document}